\documentclass[12pt, reqno]{amsart}
\usepackage{amssymb}
\input xypic
\xyoption{all}

\newtheorem{thm}{Theorem}
\newtheorem{pro}[thm]{Proposition}
\newtheorem{cor}[thm]{Corollary}
\newtheorem{lem}[thm]{Lemma}

\newtheorem{con}{Conjecture}
\theoremstyle{remark}

\newtheorem{exa}[thm]{Example}

\newcommand{\idealp}{{\mathfrak p}}

\newcommand{\idealm}{{\mathfrak m}}

\newcommand{\aut}{{\rm Aut}}

\newcommand{\derk}{{\rm Der_k}}

\newcommand{\A}{{\mathbb A}}

\newcommand{\C}{{\mathbb C}}

\newcommand{\N}{{\mathbb N}}

\newcommand{\Q}{{\mathbb Q}}

\newcommand{\tor}{\xymatrix{\ar@{-->}[r]&}}

\begin{document}

\author{Rene Baltazar}

\address{Instituto de Matematica, Estatistica e Fisica\\
  Universidade Federal do Rio Grande - FURG\\
  Santo Antonio da Patrulha, 95500-000}
\email[R.~Baltazar]{rene.baltazar@ufrgs.br}

\thanks{Research of R. Baltazar was partially supported by CAPES}

\title[On solutions for derivations]{On simple Shamsuddin derivations in two variables}
\maketitle

\begin{abstract} We study the subgroup of $k$-automorphisms of $k[x,y]$ which commute with
a simple derivation $D$ of $k[x,y].$ We prove, for example, that this subgroup
is trivial when $D$ is a Shamsuddin simple derivation. In the general case of
simple derivations, we obtain properties for the elements of this subgroup.
\end{abstract}

\section{introduction}

Let $k$ be an algebraically closed field of characteristic zero and the ring $k[x,y]$ of polynomials over $k$ in two variables. 

A $k$-\emph{derivation} $d:k[x,y]\rightarrow k[x,y]$ of $k[x,y]$ is a $k$-linear map such that $$d(ab)=d(a)b+ad(b)$$ for any $a,b \in k[x,y]$. Denoting by $\derk(k[x,y])$ the set of all $k$-derivations of $k[x,y]$. Let $d \in \derk(k[x,y])$. An ideal $I$ of $k[x,y]$ is called $d$-\emph{stable} if $d(I) \subset I$. For example, the ideals $0$ and $k[x,y]$ are always $d$-stable. If $k[x,y]$ has no other $d$-stable ideal it is called $d$-\emph{simple}. Even in the case of two variables, a few examples of simple derivations are known (see for explanation \cite{BLL}, \cite{Cec}, \cite{Now08}, \cite{BP}, \cite{Kour} and \cite{Leq}).

We denote by $\aut(k[x,y])$ the group of $k$-automorphisms of $k[x,y]$. Let $\aut(k[x,y])$ act on $\derk(k[x,y])$ by: $$(\rho,D) \mapsto \rho^{-1}\circ D \circ \rho=\rho^{-1} D \rho.$$

Fixed a derivation $d \in \derk(k[x,y])$. The isotropy subgroup,
with respect to this group action, is $$\aut(k[x,y])_D:= \{ \rho \in \aut(k[x,y]) / \rho D=D \rho\}.$$

We are interested in the following question proposed by I.Pan (see \cite{B2014}):

\begin{con} If $d$ is a simple derivation of $k[x,y]$, then $\aut(k[x,y])_d$ is finite.
\end{con}

At a first moment, in the \S 2, we show that the conjecture is true for a family of derivations, named Shamsuddin derivations (Theorem \ref{thm3.2}). For this, we use a theorem of the Shamsuddin \cite{Shams}, mentioned in \cite[Theorem 13.2.1.]{Now}, that determines a condition that would preserve the simplicity by extending, in some way, the derivation to $ R[t]$, with $ t $ an indeterminate. The reader may also remember that Y.Lequain \cite{Leq} showed that these derivations check a conjecture about the $\A_n$, the Weyl algebra over $k$.

In order to understand the isotropy of a simple derivation of the $k[x,y]$, in \S 3, we analysed necessary conditions for an automorphism to belong to the isotropy of a simple derivation. For example, we prove that if such an automorphism has a fixed point, then it is the identity (Proposition \ref{pro3.3}). Following, we present the definition of \emph{dynamical degree} of a polynomial application and thus proved that in the case $k = \C $, the elements in $ \aut(\C [x,y])_d $, with $d$ a simple derivation, has dynamical degree $1$ (Corollary \ref{cor.3.3}). More precisely, the condition dynamical degree $>1$
corresponds to exponential growth of degree under iteration, and this may be viewed as a complexity of the automorphism in the isotropy (see \cite{FM}).

\section{Shamsuddin derivation}

The main aim of this section is study the isotropy group of the a Shamsuddin derivation in $k[x,y]$. In \cite[\S 13.3]{Now}, there are numerous examples of these derivations and also shown a criterion for determining the simplicity; furthermore, Y.Lequain \cite{Leq2008} introduced an algorithm for determining when an Shamsuddin derivation is simple. However, before this, the following example shows the isotropy of an arbitrary derivation can be complicated.

\begin{exa} Let be $d=\partial_x \in \derk(k[x,y])$ and $\rho \in \aut(k[x,y])_d$. Note that $d$ is not a simple derivation; indeed, for any $u(y) \in k[y]$, the ideal generated by $u(x)$ is always invariant. Consider $$\rho(x)=f(x,y)=a_0(x)+ a_1(x)y+\ldots+a_t(x)y^t$$ $$\rho(y)=g(x,y)=b_0(x)+ b_1(x)y+\ldots+b_s(x)y^s.$$ Since $\rho \in \aut(k[x,y])_d$, we obtain two conditions:

\textbf{1)} $\rho (d(x))=d(\rho (x)).$

Thus, $$1=d(a_0(x)+ a_1(x)y+\ldots+a_t(x)y^t)=d(a_0(x))+d(a_1(x))y+\ldots+d(a_t(x))y^t.$$

Then, $d(a_0(x))=1$ and $d(a_j(x))=0$, $j=1,\ldots,t$. We conclude that $\rho (x)$ is of the type $$\rho (x)=x+c_0+c_1y+\ldots+c_ty^t, \ \ c_i \in k.$$

\textbf{2)} $\rho (d(y))=d(\rho (y)).$

Analogously, $$0=d(b_0(x)+b_1(x)y+\ldots+b_s(x)y^s)=d(b_0(x))+d(b_1(x))y+\ldots+d(b_s(x))y^s.$$

That is, $b_i(x)=d_i$ with $d_i \in k$. We conclude also that $\rho (y)$ is of the type $$\rho (y)=d_0+d_1y+\ldots+d_sy^s, \ \ d_i \in k.$$

Thus, $\aut(k[x,y])_d$ contains the affine automorphisms $$(x+uy+r, uy+s),$$ with $u, r, s \in k$. In particular, $\aut(k[x,y])_d$ is not finite.

Notice that $\aut(k[x,y])_d$ contains also the automorphisms of the type $(x+p(y), y)$, with $p(y) \in k[y]$.

Now, we determine indeed the isotropy. Using only the conditions $1$ and $2$, $$\rho=(x+p(y), q(y))$$ with $p(y), q(y) \in k[y]$. However, $\rho$ is an automorphism, in other words, the determinant of the Jacobian matrix must be a nonzero constant. Thus, $|J_\rho|=q'(y)=c$, $c \in k^*$. Therefore, $\rho=(x+p(y), ay+c)$, with $p(y) \in k[y]$ and $a,b \in k$. Consequently, $\aut(k[x,y])_d$ is not finite and, more than that, the first component has elements with any degree.
\end{exa}

The following lemma is a well known result.

\begin{lem} Let $R$ be a commutative ring, $d$ a derivation of $R$ and $h(t) \in R[t]$, with $t$ an indeterminate. Then, we can also extend $d$ to a unique derivation $\tilde{d}$ of the $R[t]$ such that $\tilde{d}(t)=h(t)$.
\end{lem}

We will use the following result of Shamsuddin \cite{Shams}.

\begin{thm} \label{thm4}
Let $R$ be a ring containing $\Q$ and let $d$ be a simple derivation of $R$. Extend the derivation $d$ to a derivation $\tilde{d}$ of the polynomial ring $R[t]$ by setting $\tilde{d}(t) = at + b$ where $a,b \in R$. Then the following two conditions are equivalent:

$(1)$  $\tilde{d}$ is a simple derivation.

$(2)$ There exist no elements $r\in R$ such that $d(r)=ar+b$.
\end{thm}
\begin{proof} See \cite[Theorem 13.2.1.]{Now} for a demonstration in details.
\end{proof}

A derivation $d$ of $k[x,y]$ is said to be a \emph{Shamsuddin derivation} if $d$ is of the form $$d=\partial_x+(a(x)y+b(x))\partial_y,$$ where $a(x),b(x) \in k[x]$.

\begin{exa} Let $d$ be a derivation of $k[x,y]$ as follows
\[d=\partial_x+(xy+1)\partial_y.\]

Writing $R=k[x]$, we know that $R$ is $\partial_x$-simple and, taking $a=x$ and $b=1$, we are exactly the conditions of Theorem \ref{thm4}. Thus, we know that $d$ is simple if, and only if, there exist no elements $r \in R$ such that $\partial_x (r)=xr+1$; but the right side of the equivalence is satisfied by the degree of $r$. Therefore, by Theorem \ref{thm4}, $d$ is a simple derivation of $k[x,y]$.
\end{exa}

\begin{lem} \label{lem3.2} (\cite[Proposition. 13.3.2]{Now}) Let $d=\partial_x+(a(x)y+b(x))\partial_y$ be a Shamsuddin derivation, where $a(x),b(x) \in k[x]$. Thus, if $d$ is a simple derivation, then $a(x) \neq 0$ and $b(x) \neq 0$.
\end{lem}
\begin{proof} If $b(x)=0$, then the ideal $(y)$ is $d$-invariante. If $a(x)=0$, let $h(x) \in k[x]$ such that $h'=b(x)$, then the ideal $(y-h)$ is $d$-invariante.
\end{proof}

One can determine the simplicity of the a Shamsuddin derivation according the polynomials $a(x)$ and $b(x)$ (see (\cite[\S 13.3]{Now})).

\begin{thm} \label{thm3.2} Let $D \in \derk(k[x,y])$ be a Shamsuddin derivation. If $D$ is a simple derivation, then $\aut(k[x,y])_D=\{id\}$.
\end{thm}
\begin{proof}

Let us denote $\rho(x)=f(x,y)$ and $\rho(y)=g(x,y)$. Let $D$ be a Shamsuddin derivation and $$D(x)=1,$$ $$D(y)=a(x)y+b(x),$$ where $a(x),b(x) \in k[x]$

Since $\rho \in \aut(k[x,y])_D$, we obtain two conditions:

$(1)$ $\rho (D(x))=D(\rho (x)).$

$(2)$ $\rho (D(y))=D(\rho (y)).$

Then, by condition $(1)$, $D(f(x,y))=1$ and since $f(x,y)$ can be written in the form $$f(x,y)=a_0(x)+ a_1(x)y+...+a_s(x)y^s,$$ with $s\geq 0$, we obtain
\[D(a_0(x))+D(a_1(x))y+a_1(x)(a(x)y+b(x))+\ldots\]
\[+D(a_s(x))y^s+ sa_s(x)y^{s-1}(a(x)y+b(x))=1\]

Comparing the coefficients in $y^s$, 
$$D(a_s(x))=-sa_s(x)a(x),$$ which can not occur by the simplicity. More explicitly, the Lemma \ref{lem3.2} implies that $a(x)=0$. Thus $s=0$, this is $f(x,y)=a_0(x)$. Therefore $D(a_0(x))=1$ and $f=x+c$, with $c$ constant.

Using the condition $(2)$,
$$\begin{array}{rl}
D(g(x,y))&=\rho (a(x)y+b(x))\\
&=\rho(a(x))\rho(y)+\rho(b(x))\\
&=a(x+c)g(x,y)+b(x+c)
\end{array}$$

Writing $g(x,y)=b_0(x)+ b_1(x)y+\ldots+b_t(x)y^t$; wherein, by the previous part, we can suppose that $t>0$, because $\rho$ is a automorphism. Thus
$$\begin{array}{rl}
a(x+c)g(x,y)+b(x+c) = &D(b_0(x))+D(b_1(x))y+b_1(x)(a(x)y+b(x))+\\
&+\ldots+D(b_t(x))y^t+tb_t(x)y^{t-1}(a(x)y+b(x)).
\end{array}$$

Comparing the coefficients in $y^t$, we obtain
$$D(b_t(x))+tb_t(x)a(x)=a(x+c)b_t(x)$$

Then $D(b_t(x))=b_t(x)(-ta(x)+a(x+c))$. In this way, $b_t(x)$ is a constant and, consequently, $a(x+c)=ta(x)$. Comparing the coefficients in the last equality, we obtain $t=1$ and then $b_1(x)=b_1$ a constant. Moreover, if $a(x)$ is not a constant, since $a(x+c)=a(x)$, is easy to see that $c=0$. Indeed, if $c \neq 0$ we obtain that the polynomial $a(x)$ has infinite distinct roots. If $a(x)$ is a constant, then $a(x)$ $D$ is not a simple derivation (a consequence of \cite[Lemma.2.6 and Theorem.3.2]{Leq2008}; thus, we obtain $c=0$.

Note that $g(x,y)=b_0(x)+ b_1y$ and, using the condition $(2)$ again,
$$\begin{array}{rl}
D(g(x,y))=&D(b_0(x))+b_1(a(x)y+b(x))\\
&=a(x)(b_0(x)+b_1y)+b(x).
\end{array}$$

Considering the independent term of $y$,
\begin{equation}
D(b_0(x))=b_0(x)a(x)+b(x)(1-b_1) \label{eq1.1}
\end{equation}

If $b_1 \neq 1$, we consider the derivation $D'$ such that $$D'(x)=1, \ \ D'(y)=a(x)y+b(x)(1-b_1).$$

In  \cite[Proposition. 13.3.3]{Now}, it is noted that $D$ is a simple derivation if and only if $D'$ is a simple derivation. Furthermore, by the Theorem \ref{thm4}, there exist no elements $h(x)$ in $K[x]$ such that $$ D(h(x))=h(x)a(x)+b(x)(1-b_1):$$

what contradicts the equation (\ref{eq1.1}). Then, $b_1=1$ and $D(b_0(x))=b_0(x)a(x)$, since $D$ is a simple derivation we know that $a(x) \neq 0$, consequently $b_0(x)=0$. This shows that $\rho=id$.
\end{proof}

\section{On the isotropy of the simple derivations}

The purpose of this section is to study the isotropy in the general case of a simple derivation. More precisely, we obtain results that reveal some characteristics of the elements in $\aut(k[x,y])_D$. For this, we use some concepts presented in the previous sections and also the concept of dynamical degree of a polynomial application.

In \cite{BP}, which was inspired by \cite{BLL}, we introduce and study a general notion of solution associated to a Noetherian differential $k$-algebra and its relationship with simplicity.

The following proposition geometrically says that if an element in the isotropy of a simple derivation has fixed point then it is the identity automorphism.

\begin{pro} \label{pro3.3} Let $D \in \derk(k[x_1,...,x_n])$ be a simple derivation and $\rho \in \aut(k[x_1,...,x_n])_D$ an automorphism in the isotropy. Suppose that there exist a maximal ideal $\idealm \subset k[x_1,...,x_n]$ such that $\rho(\idealm)=\idealm$, then $\rho =id$.

\end{pro}
\begin{proof} Let $\varphi$ be a solution of $D$ passing through $\idealm$ (see \cite[Definition.1.]{BP}). We know that $\dfrac{\partial} { \partial t} \varphi= \varphi D$ and $\varphi ^{-1}((t))=\idealm$. If $\rho \in \aut(k[x_1,...,x_n])_D$, then $$\dfrac{\partial} { \partial t}\varphi  \rho=\varphi  D  \rho=\varphi  \rho D.$$

In other words, $\varphi \rho$ is a solution of $D$ passing through $\rho^{-1}(\idealm)=\idealm$. Therefore, by the uniqueness of the solution (\cite[Theorem.7.(c)]{BP}), $\varphi \rho= \varphi$. Note that $\varphi$ is one to one, because $k[x_1,...,x_n]$ is $D$-simple and $\varphi$ is a nontrivial solution. Then, we obtain that $\rho =id$.

\end{proof}

F. Lane, in \cite{Lane}, proved that every $k$-automorphism $\rho$ of $k[x,y]$ leaves a nontrivial proper ideal $I$ invariant, over an algebraically closed field; this is, $\rho (I)\subseteq I$. Em \cite{Shams82}, A. Shamsuddin proved that this result does not extend to $k[x,y,z]$, proving that the $k$-automorphism given by $\chi(x)=x+1$, $\chi(y)=y+xz+1$ e $\chi(z)=y+(x+1)z$ has no nontrivial invariant ideal.

Note that, in addition, $\rho$ leaves a nontrivial proper ideal $I$ invariant if and only if $\rho (I)=I$, because $k[x,y]$ is Noetherian. In fact, the ascending chain $$I \subset \rho^{-1}(I)\subset \rho^{-2}(I) \subset \ldots  \subset \rho^{-l}(I)\subset \ldots  $$ must stabilize; thus, there exists a positive integer $n$ such that $\rho^{-n}(I)=\rho^{-n-1}(I)$, then $\rho(I)=I$. 

Suppose that $\rho \in \aut(k[x,y])_D$ and $D$ is a simple derivation of $k[x,y]$. If this invariant ideal $I$ is maximal, by the Proposition \ref{pro3.3}, we have $\rho=id$.

Suppose that $I$, this invariant ideal, is radical. Let $I=(\idealm_1\cap \ldots \cap \idealm_s) \cap (\idealp_1 \cap \ldots \cap \idealp_t)$ be a primary decomposition where each ideal $\idealm_i$ is a maximal ideal and $\idealp_j$ are prime ideals with height $1$ such that $\idealp_j=(f_j)$, with $f_j$ irreducible (by  \cite[Theorem 5.]{Kaplan}). If $$\rho (\idealm_1\cap \ldots \cap \idealm_s)=\idealm_1\cap \ldots \cap \idealm_s,$$ we claim that $\rho^N$ leaves invariant one maximal ideal for some $N \in \N$: suppose $\idealm_1$ this ideal. Indeed, we know that $\rho(\idealm_1)\supset \idealm_1\cap \ldots \cap \idealm_s $, since $\rho(\idealm_1)$ is a prime ideal, we deduce that  $\rho(\idealm_1) \supseteq \idealm_i$, for some $i=1,\ldots,s$ (\cite[Prop.11.1.(ii)]{AM}). Then, $\rho(\idealm_1) = \idealm_i$; that is, $\rho^N$ leaves invariant the maximal ideal $\idealm_1$ for some $N \in \N$. Thus follows from  Proposition \ref{pro3.3} that $\rho^N=id$.

Note that $\rho(\idealp_1 \cap \ldots \cap \idealp_t)=\idealp_1 \cap \ldots \cap \idealp_t$. In fact, writing $\idealp_1 \cap \ldots \cap \idealp_t=(f_1\ldots f_t)$, with $f_i$ irreducible, we can choose $h \in \idealm_1\cap \ldots \cap \idealm_s$ such that $\rho(h) \not\in \idealp_1$. We observe that there exists $h$. Otherwise, we obtain $\idealm_1\cap \ldots \cap \idealm_s \subset \idealp_1$, then $\idealp_1 \supseteq \idealm_i$, for some $i=1,\ldots,s$ (\cite[Prop.11.1.(ii)]{AM}): a contradiction. Thus, since $hf_1\ldots f_t \in I$, we obtain $\rho(h)\rho(f_1)\ldots \rho(f_t) \in I \subset \idealp_1$. Therefore, $\rho(f_1 \ldots f_t) \in \idealp_1$. Likewise, we conclude the same for the other primes $\idealp_i$, $i=1,\ldots,t$. Finally, $\rho(\idealp_1 \cap \ldots \cap \idealp_t)=\idealp_1 \cap \ldots \cap \idealp_t$.

With the next corollary, we obtain some consequences on the last case.

\begin{cor} Let $\rho \in \aut(k[x,y])_D$, $D$ a simple derivation of $k[x,y]$ and $I=(f)$, with $f$ reduced, a ideal with height $1$ such that $\rho(I)=I$. If $V(f)$ is singular or some irreducible component $C_i$ of $V(f)$ has genus greater than two, then $\rho$ is a automorphism of finite order.
\end{cor}

\begin{proof} Suppose that $V(f)$ is not a smooth variety and let $q$ be a singularity of $V(f)$. Since the set of the singular points is invariant by $\rho$, then there exist $N \in \N$ such that $\rho^N(q)=q$. Using that $\rho \in \aut(k[x,y])_D$, we obtain, by Proposition \ref{pro3.3}, $\rho^N=id$.

Let $C_i$ be a component irreducible of $V(f)$ that has genus greater than two. Note that there exist $M \in \N$ such that $\rho^M(C_i)=C_i$. By \cite[Thm. Hunvitz, p.241]{Fark}, the number of elements in $\aut(C_i)$ is finite; in fact, $\#( \aut(C_i)) < 84(g_i-1)$, where $g_i$ is the genus of $C_i$. Then, we deduce that $\rho$ is a automorphism of finite order.

\end{proof}

We take for the rest of this section $k=\C.$

Consider a polynomial application $f(x,y)=(f_1(x,y),f_2(x,y)):\C^2\rightarrow \C^2$ and define the degree of $f$ by $ \deg (f):=\max(\deg(f_1),\deg(f_2))$. Thus we may define the
dynamical degree (see \cite{BD}, \cite{FM}, \cite{Silv}) of $f$ as $$\delta(f):= \displaystyle\lim_{n \to \infty}(\deg(f^n))^{\frac{1}{n}}.$$

\begin{cor} \label{cor.3.3}
If $\rho \in \aut(\C[x,y])_D$ and $D$ is a simple derivation of $\C[x,y]$, then $\delta(\rho)=1$.
\end{cor}
\begin{proof}
Suppose $\delta(\rho)>1$. By \cite[Theorem 3.1.]{FM}, $\rho^n$ has exactly $\delta(\rho)^n$ fix points counted with multiplicities. Then, by Proposition \ref{pro3.3}, $\rho=id$, which shows that dynamical degree of $\rho$ is 1.
\end{proof}

\renewcommand{\abstractname}{Acknowledgements}
\begin{abstract}
I would like to thank Ivan Pan for his comments and suggestions.
\end{abstract}

\
\

\end{document}